\documentclass{amsproc}

\usepackage{amssymb}

\usepackage{graphicx}

\usepackage[cmtip,all]{xy}

\usepackage{hyperref}
\usepackage{mathrsfs}
\usepackage{enumerate}

\newcommand{\N}{{\mathbb N}}

\newcommand{\R}{{\mathbb R}}

\newcommand{\sfG}{\mathsf{G}}

\newcommand{\sfJ}{\mathsf{J}}

\newcommand{\sfX}{\mathsf{X}}
\newcommand{\subspX}{\sfX}

\newcommand{\Hom}{\operatorname{Hom}}

\newcommand{\Tot}{\operatorname{Tot}}

\newcommand{\supp}{\operatorname{supp}}

\newcommand{\calE}{\mathcal{E}}

\newcommand{\calA}{\mathcal{A}}

\newcommand{\calF}{\mathcal{F}}

\newcommand{\calJ}{\mathcal{J}}

\newcommand{\calC}{\mathcal{C}}
\newcommand{\calO}{\mathcal{O}}

\newcommand{\scrA}{\mathscr{A}}

\newcommand{\scrC}{\mathscr{C}}
\newcommand{\scrD}{\mathscr{D}}
\newcommand{\scrE}{\mathscr{E}}
\newcommand{\scrH}{\mathscr{H}}
\newcommand{\scrI}{\mathscr{I}}
\newcommand{\scrJ}{\mathscr{J}}

\newcommand{\im}{\operatorname{im}}
\newcommand{\Spec}{\operatorname{Spec}}
\newcommand{\hatotimes}{\hat{\otimes}}
\newcommand{\hscrC}{\hat{\scrC}}

\newcommand{\doublegroupoid}[4]{\xymatrix{#1  \ar@<2pt>[d] \ar@<-2pt>[d] & #2 \ar@<2pt>[l] 
\ar@<-2pt>[l] \ar@<2pt>[d] \ar@<-2pt>[d] \\ #3  & #4 \ar@<2pt>[l] \ar@<-2pt>[l]}}

\newtheorem{theorem}{Theorem}[section]
\newtheorem{lemma}[theorem]{Lemma}
\newtheorem{proposition}[theorem]{Proposition}
\newtheorem{conjecture}[theorem]{Conjecture}

\theoremstyle{definition}
\newtheorem{definition}[theorem]{Definition}

\theoremstyle{remark}
\newtheorem{remark}[theorem]{Remark}

\numberwithin{equation}{section}

\begin{document}

\title{Localization in Hochschild homology}


\author[M.J. Pflaum]{Markus J. Pflaum}
\address{Department of Mathematics,
  University of Colorado,
  Boulder CO 80309,
  USA}
\email{markus.pflaum@colorado.edu}
\thanks{The author gratefully acknowledges support for this work by the National Science Foundation under
  Grant No.~DMS-1105670 and by the Simons Foundation under Grant No.~359389.}

\date{\today}

\dedicatory{Dedicated to the 40th birthday of cyclic cohomology}

\begin{abstract}
%
Localization methods are ubiquitous in cyclic homology theory,
but vary in detail and are used in different  scenarios.
In this paper we will elaborate on a common feature of localization methods
in noncommutative geometry, namely sheafification of the algebra under consideration
and reduction of the computation to the stalks of the sheaf. The novelty of our approach
lies in the methods we use which mainly stem from real instead of complex algebraic geometry. 
We will then indicate how this method can be used to determine the
Hochschild homology theory of more complicated algebras out of simpler ones.

\end{abstract}

\maketitle

%

\section*{Introduction}
Since the invention of cyclic homology theory 40 years ago, powerful methods have been
developed to compute the Hochschild and cyclic homology of a given algebra.
They range from
passing to more tractable resolutions of the algebra like the Connes--Koszul resolution
\cite{ConNDG} over to the localization methods by Loday \cite{LodCHS}, Burghelea \cite{BurCHGR}, Brylinski \cite{BryCLHH},
Weibel \cite{WeiCHS}, Connes--Moscovici \cite{ConMosCCNCHG}, Teleman \cite{TelMHH,TelLHHCFA},
Nistor \cite{NisCCCPAG}, Brasselet-Pflaum \cite{BraPflHAWFSS} and others.
Localization essentially amounts to the sheafification of the Hochschild chain
complex of the algebra under consideration. The idea is to reduce the computation of the ``global''
Hochschild or cyclic homology to ``local'' ones at the stalks. In many cases,
the homologies of the stalks are easier to compute. One then has to glue the resulting homologies
to a global one and thus  obtains under reasonable assumptions the homology of the originally given algebra. 
In the case of algebraic varieties,
a sheafification approach has been originally suggested by Loday \cite[3.4]{LodCHS} and worked out by
Weibel--Geller \cite{WeiGelEDHCH} and Weibel \cite{WeiCHS} to define the Hochschild homology
of algebraic varieties and schemes. Note that in the case of schemes using a sheaf theoretic
version of the Hochschild chain complex is indispensible since the algebra of global
sections of the structure sheaf might be trivial, hence so would be its Hochschild homology
if it were defined solely through the algebra of global sections.
In the case of smooth manifolds or stratified spaces endowed with a structure sheaf of smooth functions on it, Teleman \cite{TelMHH,TelLHHCFA} used smooth partitions of unity and cutoff functions around the
diagonal  to construct a chain complex quasi-isomorphic to the
Hochschild chain complex. Referring to its construction the resulting 
complex has been called the \emph{diagonal complex}. In many situations it
is significantly easier to handle than the original Hochschild chain complex. 

Here we propose a sheafification method for the Hochschild chain complex of the global
section algebra of a differentiable space or more generally of a sheaf of bornological
algebras over a differentiable space. In a certain sense it can be understood as a combination
of the approaches by Weibel and Teleman. In the first section of this paper we provide a rather detailed
introduction to differentiable spaces from the point of view of real algebraic geometry. We
mostly follow \cite{NavSanDS} and supplement it by material from \cite{PflAGSSS}.
In the second section we prove a localization result related to but stronger than the ones from \cite{TelMHH,TelLHHCFA,BraPflHAWFSS}
and \cite{PflPosTanHochschild}. More precisely, we construct a fine sheaf complex which we call the 
diagonal sheaf complex and which in degree $k$ consists of relative Whitney fields concentrated around the
diagonal of the $k$-the power of the differentiable space. The chain complex of global sections of the diagonal sheaf complex turns out
to be quasi-isomorphic to the original Hochschild chain complex.
This allows to reduce the computation of the Hochschild homology of the original algebra to  the computation of the Hochschild homologies of the stalks which are often much easier to determine. 

Applications of our localization method are given in the last section and an
outlook is provided how it can help unravelling the singularity structure of differentiable stratified spaces.
Finally, localization in cyclic homology is briefly addressed.

\section{Preliminaries from real algebraic geometry}
\label{sec-preliminaries}
In this section we introduce a category of commutative locally $\R$-ringed
spaces which we call differentiable spaces and for which we want to provide
localization methods to examine their Hochschild homology. In addition, we
recall some notions from the theory of Whitney fields needed in the sequel
to construct localized chain complexes. 

Locally, the reduced versions of the ringed spaces we study look like a
closed subset 
$X\subset \R^n$ of some Euclidean space with structure sheaf $\calC^\infty_{|X}$
given by the sheaf of smooth functions restricted to $X$. 
More precisely, for every relatively open subset $U\subset X$ 
the section space $\calC^\infty_{|X} (U)$ coincides with the quotient algebra
$\calC^\infty (\widetilde{U})/\calJ_X (\widetilde{U})$ where
$\widetilde{U} \subset \R^n$ is an open subset such that $\widetilde{U}\cap X = U$ and
where $\calJ_X (\widetilde{U})$ denotes the closed ideal of all smooth functions
on $\widetilde{U}$ vanishing on $X$. 
For any unital commutative $\R$-algebra $A$ let $\Spec_\R (A)$
be its \emph{real spectrum} that is the space  $\Hom_{\R\textup{-alg}} (A,\R)$ of
algebra homomorphisms from $A$ to $\R$ endowed with the Gelfand topology.
The canonical projection
$p : \calC^\infty (\R^n) \to \calC^\infty (X)$ then induces a closed embedding of real spectra
$p^* : \Spec_\R ( \calC^\infty (X)) \to \Spec_\R \big( \calC^\infty (\R^n) \big)$ with
image being the subset $X$ 
under the identification $\Spec_\R (\calC^\infty (\R^n) \cong \R^n$.

To allow for non-reduced versions of the local models we
follow \cite{NavSanDS} and  call a commutative locally $\R$-ringed space $(X,\calO)$ an
\emph{affine differentiable space} if the following conditions hold true.
\begin{enumerate}[(DS1)]
\item\label{ite:ADSembedding}
  There exists a homeomorphism  $\varphi :X \to \subspX \subset \R^n$  onto a closed subspace
  $\subspX\subset \R^n$.
\item\label{ite:ADSepimorphism}
   There exists an epimorphism of sheaves 
   $\varphi^\sharp: \calC^\infty_{\R^n} \to \varphi_*\calO $. 
\item\label{ite:ADSspectra}
  The embedding of real spectra 
  \[ (\varphi^\sharp_{\R^n})^*: \Spec_\R (\calO(X)) \to \Spec_\R \big(\calC^\infty(\R^n)\big) \cong \R^n\]
  induced by the
  projection $\varphi^\sharp_{\R^n} : \calC^\infty(\R^n) \to \calO (X)$ has image $\subspX = \varphi (X)$. 
\item\label{ite:ADSlocalization}
  For each open $U\subset X$, the section algebra $\calO (U)$ coincides with the localization
  of $\calO (X)$ by $S_U := \{ g\in \calO (X) \mid (\varphi^\sharp_{\R^n})^*(v) \neq 0 \text{ for all } v \in \subspX \}$. 
\end{enumerate}
Conditions (DS\ref{ite:ADSembedding}) and (DS\ref{ite:ADSepimorphism}) entail  that the pair
$(\varphi,\varphi^\sharp): (X,\calO ) \to \big( \R^n, \calC^\infty \big)$ is a morphism, 
more precisely even a closed embedding of locally $\R$-ringed spaces.
Following \cite[Sec.~1.3]{PflAGSSS}, we call a closed embedding
$(\varphi,\varphi^\sharp): (X,\calO ) \to \big( \R^n, \calC^\infty \big)$
such that the conditions (DS\ref{ite:ADSspectra}) and (DS\ref{ite:ADSlocalization}) hold true
a \emph{singular chart} of the affine differentiable space $(X,\calO )$.  We will usually write
$\varphi :X \to \subspX \subset \R^n$ for the underlying map to denote that its image is $\subspX$.
In case the affine differentiable space $(X,\calO )$ is \emph{reduced} meaning that $\calO$ is a
sheaf of real valued continuous functions on $X$,  the sheaf morphism $\varphi^\sharp$ is
uniquely determined by $\varphi$ and given by pullback with $\varphi$.
More precisely, over $V\subset \R^n$ open it has the form
\[
  \varphi^\sharp_V: \calC^\infty (V) \to \calO (\varphi^{-1} (V)),
  \quad f\mapsto f\circ \varphi|_{\varphi^{-1} (V)} \ . 
\]
A \emph{reduced affine differentiable space} is therefore exactly  
a commutative locally $\R$-ringed space $(X,\calO )$ isomorphic to
$\big( \subspX , \calC^\infty_{|\subspX} \big)$ for some closed $\subspX\subset \R^n$.
In the reduced case we often just call the embedding
$\varphi :X \to \subspX \subset \R^n$  inducing the 
isomorphism  a \emph{singular chart} for $(X,\calC^\infty )$. Moreover, we denote the
structure sheaf of a reduced (affine) differentiable space by $\calC^\infty_X$ instead of
$\calO$ and call the  elements of a section space $\calC^\infty_X (U) $ the
\emph{smooth functions} on the open set $U\subset X$. 

Gluing together affine differentiable spaces gives rise to the category we are interested in.
More precisely, a commutative locally $\R$-ringed space $(X,\calO)$ is called a \emph{differentiable space} if
it possesses an \emph{atlas} consisting of (local) singular charts that is if there exists a family
$(\varphi_i,\varphi^\sharp_i)_{i\in I}$ of continuous maps $\varphi_i :U_i \to \R^{n_i}$
and sheaf morphisms $\varphi^\sharp_i : \calC^\infty_{\R^{n_i}} \to (\varphi_i)_* \calO_{|U_i}$
with the following properties.
\begin{enumerate}[(DS1)]
\setcounter{enumi}{4}
\item The family $(U_i)_{i\in I}$ is an open cover of $X$.
\item Each of the maps    $\varphi_i :U_i \to \R^{n_i}$ is a homeomorphism onto a closed
      subset of $\R^{n_i}$.
\item The commutative locally $\R$-ringed space $(U_i,\calO_{|U_i})$ is an affine differentiable space
      with singular chart
      $(\varphi_i,\varphi^\sharp_i) : (U_i,\calO_{|U_i}) \to \big( \R^{n_i} ,\calC^\infty_{\R^{n_i}}\big)$
      for every $i\in I$. 
\end{enumerate}   

\begin{remark}
  In the definition of a differentiable space, the underlying topological space need not necessarily
  be Hausdorff. We do not need that generality here, so for reasons of convenience  we assume for the rest
  of the paper that $X$ is Hausdorff.  
  As in the affine case, a differentiable space will be called \emph{reduced} if its structure sheaf consists of
  continuous functions on the underlying topological space.
  For more details on differentiable spaces see the monograph \cite{NavSanDS}.
\end{remark}

Next we turn differentiable spaces into a category. To this end we define a morphism between differentiable spaces
$(X,\calO_X)$ and  $(Y,\calO_Y)$ as a morphism $(\psi,\psi^\sharp) : (X,\calO_X) \to (Y,\calO_Y) $
in the category of locally $\R$-ringed spaces. In case both  $(X,\calC^\infty_X)$ and  $(Y,\calC^\infty_Y)$  are reduced,
$\psi^\sharp$ is given over $V\subset Y$ open by the pullback
$\psi^\sharp_V = (\psi|_V)^*$.  
We say in this situation that $\psi$ is \emph{smooth}. 
In other words this means that pullback by $\psi$ maps smooth functions to smooth functions. 

A particularly interesting class of non-reduced affine differentiable spaces is  given by the Whitney fields
over a closed subset of some Euclidean space. The interpretation of Whitney fields as forming the structure sheaf
of a non-reduced affine differentiable space appears to be new, but quite helpful for our purposes. 
Originally, Whitney fields played an important role in the extension and composition theory
of smooth functions \cite{MalIDF,TouIFD,BieMilPawCDF}. They can be defined
for any affine differentiable space $X$ with a fixed singular chart.
For simplicity, we assume that $X$ is embedded as a closed subspace in some Euclidean space $\R^n$ which in
other words means that the singular chart is the identical embedding.
By a  Whitney field on a relatively open subset $U\subset X$  we then understand  a family
$F= (F_\alpha)_{\alpha\in \N^n}$ of continuous functions $F_\alpha :U \to \R$ for which there exists
a smooth $f:\widetilde{U}\to \R$ defined on some open set $\widetilde{U} \subset \R^n$ such that
$U = \widetilde{U}\cap X$ and such that  for all $\alpha \in \N^n$ the restriction of the
partial derivative $\partial^\alpha f = \frac{\partial^{|\alpha|} f}{\partial^\alpha x}$ coincides with $F_\alpha$ that is if $F_\alpha = \partial^\alpha f |_U$.
We denote the space of Whitney fields over $U$ by $\calE^\infty (U)$ and the canonical map
$\calC^\infty (U) \to \calE^\infty (U)$, $f \mapsto (\partial^\alpha f)_{\alpha\in \N^n}$ by $\sfJ^\infty$.
The latter is sometimes called the \emph{jet map}. 
By Whitney's extension theorem there is an intrinsic definition of the concept of a Whitney field, but we do not
need that here and therefore refer the interested reader to \cite{MalIDF,TouIFD}.
The space of Whitney fields $\calE^\infty (U)$ carries a unique structure of an algebra over $\R$
such that the jet map $\sfJ^\infty : \calC^\infty (\widetilde{U}) \to \calE^\infty (U)$ becomes a morphism of algebras. 
Its kernel then is an ideal in $\calC^\infty (\widetilde{U})$ which we denote by $\calJ^\infty_X (\widetilde{U})$ or
$\calJ^\infty (X; \widetilde{U})$ and call it the ideal of smooth functions on $\widetilde{U}$ which are \emph{flat} on $X$.
Note that $\calJ^\infty_X (\widetilde{U})$
is contained in the vanishing ideal $\calJ_X(\widetilde{U})$. When $U$ runs through the open sets of $X$,
the assignment $U\mapsto \calE^\infty (U)$ becomes a fine sheaf on $X$, the sheaf of \emph{Whitney fields} on
$X$ with respect to the embedding $X\hookrightarrow \R^n$. The pair $\big( X, \calE^\infty\big)$ then becomes
an affine differentiable space which in general is not reduced. 

There is a relative version of Whitney fields which allows to get rid of the existence of a global singular chart.
To this end let $(X,\calO_X)$ be a differentiable space and $(Y,\calC^\infty_Y)$ a closed differentiable subspace
which means that $Y\subset X$ is closed and that $ \calC^\infty_Y$ is the restriction of the sheaf $\calC^\infty_X$
to $X$. Let  $\calC^\infty (X;Y)$ be the ideal in the algebra of smooth functions on $X$ for which there exists 
an atlas $(\varphi_i)_{i\in I}$ of singular charts $\varphi_i : U_i \to A_i \subset \R^{n_i}$ of $X$ such that a smooth
function $f :X \to \R$ lies in $\calC^\infty (X;Y)$ if and only if for each $i$ there is a smooth function
$f_i : \widetilde{U}_i \to \R$ defined on an open neighborhood $\widetilde{U}_i$ of $A_i$ in $\R^{n_i}$ such
that $f_i$ is flat on $\varphi_i (Y\cap U_i)$ and such that $f|_{U_i} = f_i \circ \varphi_i$. One can check that
if this condition is fulfilled in one atlas of singular charts for $X$, then it is so too in any other, hence the
definition of $\calC^\infty (X;Y)$ does not depend on the choice of an atlas. 
Following \cite{BieMilPawCDF}, we call $\calC^\infty (X;Y)$ the algebra of smooth functions on $X$ \emph{flat}
on $Y$. It is straightforward to check that associating to each open subset $U\subset X$ the algebra
$\calC^\infty (U;Y):= \calC^\infty (U;Y\cap U)$ gives rise to a fine sheaf on $X$ which actually even is an
ideal sheaf in $\calC^\infty_X$. 
Let us describe the section space $\calC^\infty (U;Y)$ when $U$ is the domain of some singular chart
$\varphi : U \to \R^n$ and $U\cap Y \neq \emptyset$. Choose $\widetilde{U}\subset \R^n$ open such that
$U = \widetilde{U} \cap A$ where $A= \varphi (X)$. Put $B =\varphi (Y)$.
The intersection $\calJ^\infty_B (\widetilde{U}) \cap \calJ_A (\widetilde{U})$ then is an ideal in
$\calJ^\infty_B (\widetilde{U})$ giving rise to a natural isomorphism of algebras 
\[
  \calC^\infty (U;Y)\cong\calJ^\infty_B (\widetilde{U})\, \big/ \, \calJ^\infty_B(\widetilde{U})\cap\calJ_A (\widetilde{U}) \ .
\]
The corresponding algebra of \emph{Whitney fields} on $Y$ \emph{relative} $X$ can now be defined and is given
as the quotient algebra
\[ \calE^\infty (Y;X) := \calC^\infty (X) / \calC^\infty (X;Y) \ .\]

Finally in this section we state a topological result whose proof is immediate from the fact that  
the algebra of smooth functions on an open set of some Euclidean space is a nuclear Fr\'echet space,
see e.g.~\cite{TreTVSDK}.
 
\begin{proposition}
\label{Prop:FrechetStructuresOnFunctionSpaces}
For $X \subset \R^n$ locally closed, and $Y \subset X $ a relatively closed subset the 
spaces $\calC^\infty (X)$, $\calE^\infty (X)$, $\calC^\infty (X,Y)$, and $\calE^\infty (Y,X)$ all carry in a 
natural way the structure of a Fr\'echet space which they inherit 
as subquotients from the Fr\'echet space $\calC^\infty (U)$, where $U\subset \R^n$ is open such that
$X \subset U$ is closed. 
\end{proposition}

\section{The diagonal complex}
\label{sec:diagonal-complex}
Let $(X,\calC^\infty)$ be a reduced differentiable space as defined in 
Section \ref{sec-preliminaries}, and $A = \calC^\infty (X)$ be the
algebra of smooth functions on $X$. Since $A$ carries the structure of a
nuclear Fr\'echet algebra, all tensor product topologies on the tensor powers $A^{\otimes k+1}$
coincide, hence the completed tensor product  
\[
  C_k (A) := A^{\hat{\otimes} {k+1} } = A \hat{\otimes} A\hat{\otimes}  \ldots \hat{\otimes} A, \quad k \in \N
\] 
is uniquely determined. Moreover, $C_k(A)$ coincides with
$\calC^\infty (X^k)$ since for open $U\subset \R^n$ and $V\subset \R^m$ there
is a natural identification
$\calC^\infty (U\times V) \cong \calC^\infty (U) \hat{\otimes} \calC^\infty (V) $, see \cite{TreTVSDK}.
The Hochschild boundary 
\begin{displaymath}
  b:\:C_k (A) \rightarrow C_{k-1} (A), \: a_0 \otimes \ldots \otimes a_k \mapsto
  \sum_{0\leq i \leq k}  (-1)^i \, b_i \big(  a_0 \otimes \ldots \otimes a_k \big) , 
\end{displaymath}  
where   
\begin{equation}
\label{eq:Hochschildbdry}
  b_i \big(  a_0 \otimes \ldots \otimes a_k \big) =
  \begin{cases}
    a_0 \otimes \ldots \otimes a_i  a_{i+1} \otimes \ldots \otimes a_k & \text{ for $0\leq i<k$},\\
    a_k a_0 \otimes \ldots \otimes a_{k-1}  & \text{ for $i=k$},
  \end{cases}
\end{equation}
then is continuous for each $k\in \N^*$, and gives rise to the  
(continuous) Hochschild complex $\big( C_\bullet (A) , b \big) $; 
see \cite{TayHCTA} for details on the Hochschild homology of Fr\'echet algebras.
The homology of the complex $\big( C_\bullet (A) , b \big) $ is the (continuous) Hochschild homology of $A$ and 
is to be denoted by $HH_\bullet (A)$. 
In the following, we will construct a chain complex $\big( E_\bullet (A) , \beta \big) $,
called the diagonal complex, 
which is quasi-isomorphic to $\big( C_\bullet (A) , b \big) $ and which can be interpreted 
as the complex of global sections of a fine sheaf complex 
$\big( \scrE_\bullet , \beta \big) $ over the space $X$. The diagonal complex originally
goes back to the work of Teleman \cite{TelMHH} on an alternative method to determine 
the  Hochschild homology of the algebra of smooth functions on a smooth manifold. 
It has been used by Brasselet--Pflaum \cite{BraPflHAWFSS} to compute
the Hochschild homology of algebras of Whitney fields over subanalytic sets. 

To define the diagonal complex, consider for every $k\in \N$ the closed ideal 
\[
  J_k (X) := \calC^\infty (X^{k+1}, \Delta^{k+1}_X ) \subset \calC^\infty (X^{k+1})
\]
of smooth functions on $X^{k + 1}$ flat on the diagonal 
\[
  \Delta_X^{k+1} := \big\{ (x_0,\ldots , x_k) \in X^{k+1} \mid x_0 = \ldots =x_k \big\} \ .
\]  
In other words, 
$J_k (X) $ is the set of all $f\in \calC^\infty (X^{k+1})$ such that for every singular chart 
$\varphi: U \rightarrow \sfX \subset \R^n$ over an open  $U \subset X$ the function
\[
  \varphi_*f : \: \sfX^{k+1} \rightarrow \R , \: (v_0,\ldots , v_k) \mapsto 
  f \big( \varphi^{-1} (v_0) , \ldots ,  \varphi^{-1} (v_k) \big) 
\]
coincides with the restriction $g|_{\sfX^{k+1}}$ of a function $g\in \calJ^\infty (\Delta_\sfX^{k+1}; \widetilde{U}^{k+1})$,
where $\widetilde{U} \subset \R^n$ is an open neighborhood of $\sfX$. 
Define $\scrE_k (X)$  as the algebra of Whitney functions on the diagonal $\Delta^{k+1}_X$ relative to the space
$X^{k+1}$ that is put
\begin{equation}
  \label{Eq:DefDiagCplx}
  \scrE_k (X) := \calE^\infty (\Delta^{k+1}_X  ; X^{k+1}) = 
  \calC^\infty (X^{k+1}) / J_k (X) \: . 
\end{equation}
\begin{lemma}
\label{Lem:HochschildBdryMapsIdealToIdeal}
  The Hochschild boundary $b$ maps $J_k (X) $ into $J_{k-1} (X)$. 
\end{lemma}
\begin{proof}
  Without loss of generality it suffices to prove that for $X$ a closed subset of some open set 
  $U$ in $\R^n$ and $ a \in \calJ^\infty (\Delta^{k+1}_X,U^{k+1})$ the image $b_i (a) $ is in 
  $\calJ^\infty (\Delta^k_X,U^k)$. To this end denote by $\partial^\alpha_j$ the differential operator 
  \[
    a_0 \otimes \ldots \otimes a_l \mapsto  
    a_0 \otimes \ldots  \otimes  (\partial^\alpha a_j) \otimes \ldots \otimes a_l , 
  \]
  where $a_0, \ldots, a_l \in \calC^\infty (U)$ and $j\leq l$.
  The higher Leibniz formula entails for  $\alpha \in \N^n$ and $i,j \in \N$ with
  $0 \leq i \leq k$ and $0\leq j \leq k-1$ that
  \begin{equation}
  \label{Eq:CommHochDiff}
    \partial^\alpha_j \big( b_i(a) \big) = b_i \big( \widetilde{a} \big),
  \end{equation}
  where 
  \begin{equation}
  \label{Eq:CommHochDiff2nd}
    \widetilde{a} = 
    \begin{cases}
      \partial^\alpha_j a  & \text{for $j < i$, $(j,i)\neq (0,k)$}\ ,\\
      \partial^\alpha_{j+1} a  & \text{for $j > i$}\ ,\\
      \sum\limits_{\beta +\gamma =\alpha} {\alpha \choose\beta} \:
      \partial^\beta_k \, \partial^\gamma_0 a  & \text{for $(j,i) =(0,k)$}\ , \\
      \sum\limits_{\beta +\gamma =\alpha} {\alpha \choose\beta} \:
      \partial^\beta_i\, \partial^\gamma_{i+1} a  & \text{for $j= i$}\ .
    \end{cases} 
  \end{equation}
  Eqs.~\eqref{Eq:CommHochDiff} and \eqref{Eq:CommHochDiff2nd} imply that for 
  a diagonal element $\Delta^k (v) := (v,\ldots ,v) \in U^k$ with $v\in X$ the equality 
  \[
    \big( \partial^\alpha_j b_i(a)\big) \, \big( \Delta^{k} (v) \big) = 
    b_i \widetilde{a} \big( \Delta^{k} (v) \big)  =  \widetilde{a} \big( \Delta^{k+1} (v) \big) = 0
  \]
  holds true, since $\widetilde{a} \in \calJ^\infty (\Delta^{k+1}_X,U^{k+1})$. 
  This entails the claim.
\end{proof}
  As a consequence of the lemma we obtain a short exact sequence of chain complexes:
  \begin{equation}
    0 \longrightarrow \big( J_\bullet (X) , b\big) \longrightarrow \big( C_\bullet (A) , b  \big)
    \longrightarrow \big( E_\bullet (A) , \beta \big)  \longrightarrow 0 \: .
  \end{equation}
  We call the quotient complex $\big( E_\bullet (A), \beta \big)$ the 
  \emph{diagonal complex} of $A$, and denote the canonical projection 
  by $p_X :C_\bullet (A) \rightarrow E_\bullet (A) $.
  
  Letting $U$ run through the open subsets of $X$ we put for each $k\in \N$
  \begin{equation}
   \label{Eq:DefShCplx}
   \scrE_k (U) := \calC^\infty (U^{k+1}) / J_k (U).
  \end{equation}
  By construction, the  $\scrE_k (U)$ form the section spaces of a presheaf on $X$ denoted by $\scrE_k$.
  Since the differential $\beta$ commutes with the restriction morphisms, 
  we even obtain a presheaf of chain complexes $\big( \scrE_\bullet , \beta \big)$.
  Its global section space coincides with the diagonal complex of $A$:
  \[
      E_\bullet (A) =  \scrE_\bullet (X) \ . 
  \]

  Even more holds true, but before that we need to introduce the following notion
  based on the concepts of regularity by Malgrange \cite{MalIDF} and 
  Tougeron \cite[Defs.~1.30 \& 4.4]{TouIFD}, 
  cf.~also \cite[Defs.~1.6.6 \& 1.6.18]{PflAGSSS} 
  \begin{definition}
    A differentiable space $(X,\calC^\infty)$ is called \emph{homologically regular} if for each point $x \in X$ there exists a singular chart 
   $\varphi : U \rightarrow \R^n$ defined on an open neighborhood of $x$ such that 
   $\varphi ( U ) \subset \R^n$ is regular in the sense of Tougeron
   and if for every $k\in \N^*$ the sets $\varphi ( U )^k$ and $\Delta^k_{\R^n}$ 
   are regularly situated in $\R^{kn}$.   
  \end{definition}

\begin{theorem}\label{thm:projection-diagonal-complex-qism}
  Let $(X,\calC^\infty)$  be a homologically regular reduced differentiable space and 
  let $A = \calC^\infty (X)$ be its global section algebra. 
  Then $\big( \scrE_\bullet , \beta \big)$ is a 
  complex of fine sheaves over $X$. If in addition $(X,\calC^\infty)$ is 
  homologically regular, then the chain map 
  $p_X :C_\bullet (A) \rightarrow \scrE_\bullet (X) = E_\bullet (A)$ onto the diagonal complex 
  is a quasi-isomorphism. 
\end{theorem}
\begin{proof}
  We proceed in several steps, and follow and extend the argument given in the proof of 
  \cite[Prop.~4.2]{BraPflHAWFSS}. 

  \textit{Step 1.}~First we show that $\scrE_k$ is a complex of fine sheaves. Let 
  $\calC^\infty_{| \Delta_X^{k+1}}$ be the restriction of the structure 
  sheaf of $X^{k+1}$ to the diagonal $\Delta_X^{k+1}$. Since 
  the structure sheaf $\calC^\infty_{X^{k+1}}$ is fine, $\calC^\infty_{| \Delta_X^{k+1}}$ is a 
  a fine sheaf, too. By construction, $\scrE_k$  is a module over the sheaf 
  $\calC^\infty_{| \Delta_X^{k+1}}$, hence $\scrE_k$ is even a sheaf and has to be fine as well. 

  \textit{Step 2.}~Next we will define for each $t > 0$ a localization chain map 
  $\Psi_{\bullet,t} : C_\bullet (A) \to C_\bullet(A) $ such that for every $c \in C_k (A)$ the chain 
  $\Psi_{\bullet,t} c $ has support in a $t$-neighborhood of the diagonal 
  $\Delta^{k+1} (X)$. To this end choose a proper embedding $\varphi$
  of $(X,\calC^\infty)$ into the colimit of ringed spaces 
  $(\R^{\infty},\calC_{\R^{\infty}}^\infty)  := 
  \lim\limits _{\underset{n\in \N}{\longrightarrow}} ( \R^n ,\calC^\infty_{\R^n})$. Such an embedding 
  exists by \cite[1.3.17 Lemma]{PflAGSSS} and \cite[Prop.\ A.3]{PflPosTanHochschild}. Denote by
  $d:X \times X\to\R$ the restriction of the Euclidean metric $d_\infty$ on $\R^{\infty}$ to $X$ 
  that is put $d(x,y)= d_\infty (\varphi (x),\varphi (y))$ for all $x,y \in X$.  
  By \cite[Cor.\ A.4]{PflPosTanHochschild}, the function $d^2$ lies in $\calC^\infty (X\times X)$. Now fix 
  a smooth function $ \varrho :\R\to [0,1]$ having support in $(-\infty ,\frac 12]$ 
  and which satisfies $\varrho (r) = 1$ for $r\leq \frac 1 3$. For $t > 0$ 
  denote by $\varrho_t$ the rescaled function $r \mapsto \varrho(\frac{r}{t})$.
  Next define functions $\Psi_{k,i,t} \in \calC^\infty (X^{k+1})$ for $k\in \N$ 
  and $i= 0, \ldots, k+1$ by
  \begin{equation}\label{eq:cutoff-functions-diagonal}
    \Psi_{k,i,t} (x_0, \ldots, x_k) =  \prod_{j=0}^{i-1} 
    \varrho_t \big( d^2(  x_j  , x_{j+1}) \big) \quad \text{where } x_0,\ldots , x_k \in X, \:  x_{k+1}:= x_0 \ .
  \end{equation} 
  Then put $\Psi_{k,t} := \Psi_{k,k+1,t}$. Since
  $\calC^\infty (X^{k+1})$ acts on $C_k(A)$ in a natural way, one obtains for each 
  $t >0$ a graded map of degree $0$
  \[
    \Psi_{\bullet,t} : C_\bullet (A) \to C_\bullet(A), \: C_k(A) \ni c \mapsto 
    \Psi_{k,t} c \: .
  \]
  One immediately checks the commutation relations
  \[
    b_j \Psi_{k,i,t} =    
    \begin{cases}
       \Psi_{k,i-1,t} b_j ,& \text{if $j < i$},\\
       \Psi_{k,i,t} b_j , & \text{if $j \geq i$},\\
    \end{cases}
  \] 
  which implies that $\Psi_{\bullet,t}$ commutes with the face maps $b_i$. 
  Hence, $\Psi_{\bullet,t}$ is a chain map. By construction, the support of 
  $\Psi_{k,t} c$ is contained in the neighborhood $U_{k+1, \frac t 2}$ of the 
  diagonal $\Delta^{k+1} (X)$, where for $\varepsilon >0$ 
  \[
    U_{k,\varepsilon} := \big\{ (x_1, \ldots ,x_k) \in X^k \mid 
    \sqrt{d^2(x_1,x_2 ) + d^2(x_2,x_3 ) + \ldots + d^2(x_k,x_1 )} < \varepsilon  \big\}
  \]
  is the so-called \emph{$\varepsilon$-neighborhood} of the diagonal 
  $\Delta^k (X) \subset X^k$.

  For later purposes let us mention at this point that the functions $\Psi_{k,i,t}$ and 
  $\Psi_{k,t}$ are restrictions of smooth functions $\Psi^\infty_{k,i,t}$ and $\Psi^\infty_{k,t}$  
  on $\big( \R^\infty \big) ^{k+1}$ which are defined analogously, just with $d_\infty$ instead 
  of $d$. 
  In the following, we will denote the restrictions of the functions
  $\Psi^\infty_{k,i,t}$ and $\Psi^\infty_{k,t}$ and to the cartesian product 
  $Y^{k+1}$ of a locally closed subset  $Y \subset \R^\infty$ 
  of $\R^\infty$,
  by the symbols $\Psi^Y_{k,i,t}$ and $\Psi^Y_{k,t}$, respectively. 

  \textit{Step 3.}~The chain map $\Psi_{\bullet,t}: C_\bullet (A) \to C_\bullet(A)$ is 
  homotopic to the identity morphism on $C_\bullet (A)$ for every $t >0$.
 
  To prove this recall first that the Hochschild chain complex of $A$ has the following 
  degeneracy maps for $k \in N$ and $i = 0, \ldots , k$:
  \begin{equation}
    \label{eq:degenarcy-maps}
    s_{k,i} : C_k (A) \rightarrow C_{k+1} (A), \; a_0 \otimes \ldots \otimes a_k \mapsto 
    a_0 \otimes \ldots \otimes a_i \otimes 1 \otimes a_{i+1} \otimes \ldots \otimes a_k \: .    
  \end{equation}
  Therefore, we can define $\calC^\infty (X)$-module maps 
  $\eta_{k,i,t} : C_k (A) \rightarrow C_{k+1} (A)$ 
  for every integer $k\geq -1$, $i=1,\cdots,k+2$ and $t >0$ by 
  \begin{equation}\label{eq:homotopy-components}
    \eta_{k,i,t} (c) :=
    \begin{cases}
      \Psi_{k+1,i,t}\cdot ( s_{k,i-1} c  )& \text{for $i\leq k+1$},\\
      0 &\text{for $i=k+2$}.
    \end{cases}
  \end{equation}
Here we have assumed $c\in C_k (A)$ and have put $C_{-1} (A) := \{ 0\}$.
For $i = 2,\cdots, k$ one then computes using the above commutation relations 
for $b_j$ and $\Psi_{k,i,t}$
\begin{equation}
  \begin{split}
   (  b \eta_{k,i,t}  +  \eta_{k-1,i,t} b) c  
   = \,  &  (-1)^{i-1} \Psi_{k,i-1,t}  c \, + \,  \Psi_{k,i-1,t} \sum_{j=0}^{i-2} \,
       (-1)^j \,  s_{k-1,i-2} b_{k,j} c \, + \\
       & + (-1)^i \Psi_{k,i,t} c \, + \Psi_{k,i,t} \sum_{j=0}^{i-1} \,
       (-1)^j \, s_{k-1,i-1} b_{k,j} c  ,
  \end{split}
\end{equation}
and for the two remaining cases $i=1$ and $i=k+1$
\begin{align} 
  ( b \eta_{k,1,t} + \eta_{k-1,1,t} b) c
      = & \, c  \, - \, \Psi_{k,1,t} c \, + \, \Psi_{k,1,t} s_{k-1,0} b_{k,0} c ,
  \\[2mm]
  ( b \eta_{k,k+1,t}  +  \eta_{k-1,k+1,t} b ) c  = & \\
    = \Psi_{k,k,t} (-1)^k c \, + \, \Psi_{k,k,t} & \sum_{j=0}^{k-1}
      \, (-1)^j \, s_{k-1,k-1} b_{k,j} c  \, + \,
      (-1)^{k+1} \Psi_{k,t} \, c  . \nonumber
\end{align}
These formulas immediately entail that the collection $\big(H_{k,t}\big)_{k\in \N}$ of maps
  \begin{displaymath}
    \begin{split}
      H_{k,t} & = \sum_{i=1}^{k+1} \, (-1)^{i+1} \, \eta_{k,i,t} : 
      C_k (A)\rightarrow C_{k+1} (A)
    \end{split}
  \end{displaymath}
  forms a homotopy between the identity  and the localization
  morphism $\Psi_{\bullet,t}$. More precisely,
  \begin{align}
  \label{EqAlgHom}
      \big( b H_{k,t} + H_{k-1,t} b \big) c & = c - \Psi_{k,t} \, c
      \quad \text{for all }c\in C_k (A)\ . 
  \end{align}     
  This finishes the third step.

  Note that by the considerations at the end of the previous step, the functions 
  $\eta_{k,i,t}$ can be also defined on the spaces $C_k(\calC^\infty (Y))$,
  where $Y \subset \R^\infty$ is again a locally closed subset.
  We write $\eta^Y_{k,i,t}$ for these maps. Observe that for each 
  locally closed $Z \subset Y$ the diagrams
  \begin{equation}
  \label{Dia1}
  \vcenter{\xymatrix{
    C_k(\calC^\infty (Y)) \ar[r]^{\Psi^Y_{k,i,t}} \ar[d] & C_{k+1} (\calC^\infty (Y)) \ar[d] \\
    C_k(\calC^\infty (Z)) \ar[r]_{\Psi^Z_{k,i,t}} & C_k(\calC^\infty (Z)) 
    }}
  \end{equation}
  and
  \begin{equation}
  \label{Dia2}
  \vcenter{\xymatrix{
    C_k(\calC^\infty (Y)) \ar[r]^{\eta^Y_{k,i,t}} \ar[d] & C_{k+1} (\calC^\infty (Y)) \ar[d] \\
    C_k(\calC^\infty (Z)) \ar[r]_{\eta^Z_{k,i,t}} & C_k(\calC^\infty (Z))
  }}
  \end{equation}
  then commute, where the vertical arrows denote the natural restriction maps. 
   
  \textit{Step 4.}~Next we study the operators 
  $K_{k,t}: C_k (A)\rightarrow C_{k+1} (A)$ defined for $t \in (0,1]$ by  
  \begin{equation}
    \label{eq:homotopies-between-t-values}
    K_{k,t} c := \int_t^1 H_{k, \frac s2 } \big( \partial_s \Psi_{k,s} c\big)ds\: .
  \end{equation}
  Since $\Psi_s$ is a chain map, the algebraic homotopy relation stated in  
  Eq.~\eqref{EqAlgHom} entails (as in Eq.~(4.9) of \cite{BraPflHAWFSS}) that for $c \in C_k(A)$
  \begin{equation}
  \label{EqTDepContHom}
  \begin{split}  
    \big( b K_{k,t} + K_{k-1,t} b \big) c \, &  =  
    \int_t^1 b H_{k, \frac s2} \big( \partial_s \Psi_{k,s} c\big) + 
    H_{k-1, \frac s2 } b \big( \partial_s \Psi_{k,s} c \big) ds  = \\
    & = \int_t^1  \partial_s \Psi_{k,s} c -  \Psi_{k,\frac s2} (\partial_s \Psi_{k,s} c) ds =  
    \Psi_{k,1} c - \Psi_{k,t} c \: ,
  \end{split}
  \end{equation}
  where we have used that $\operatorname{supp} \Psi_{k,\frac s2} \subset U_{k+1,\frac s4}$ and 
  $(\partial_s \Psi_{k,s})_{|U_{k+1,\frac s4}}=0$. 

  \textit{Step 5.} In the last step we show that $J_\bullet (X)$ is a contractible complex
  which will entail the claim that $p_X :C_\bullet (A) \rightarrow \scrE_\bullet (X)$ 
  is a quasi-isomorphism.  To this end it suffices to show by Eq.~\eqref{EqTDepContHom}
  that for $c \in J_k (X)$ the relation
  \begin{equation}\label{eq:localization-zero-limit}
    \lim_{t \searrow 0} \Psi_{k,t} c = 0 
  \end{equation}
  holds true, and that the operator  
  \begin{equation}\label{eq:homotopy-zero-limit}
    K_k : J_k (X) \to J_{k+1} (X) , \: c \mapsto \lim_{t \searrow 0} K_{k,t} c
  \end{equation}
  exists and is continuous. The limits in Equations \eqref{eq:localization-zero-limit}
  and \eqref{eq:homotopy-zero-limit} are taken with respect to the natural Fr\'echet topology
  which the space $J_k (X)$ inherits as a closed subspace of $C_k (A) = \calC^\infty (X^{k+1})$; see
  Prop.~\ref{Prop:FrechetStructuresOnFunctionSpaces}. Describing convergence of a sequence
  $(f_l)_{l\in \N} \subset J_k (X)$ is somewhat tricky, but it essentially boils down 
  to the criterion that for every compact $K \subset X$ there is an ambient Euclidean space $\R^n \supset K$
  (provided by an appropriate singular chart) and a sequence of representatives $F_l \in \calC^\infty (\R^{(k+1)n})$ of 
  $f_l$ over $\Delta^{k+1}_K$ 
  such that $(F_l)_{l\in \N}$ and all sequences of higher derivatives
  $(\partial^\alpha F_l)_{l\in \N}$ converge uniformly on $\Delta^{k+1}_K$.   
  
  Now denote by $B_n$ the closed ball in $\R^n$ of radius $n$, and put 
  $C_n := B_n \cap X$.  Then $(C_n)_{n\in \N}$ is a compact exhaustion of $X$.
  Consider for locally closed $Y \subset \R^\infty$ the maps
  \begin{displaymath}
  \begin{aligned}
    \mu^Y_k  : & \, J_k(Y) \times [0,1] \to  J_k (Y), \quad
                  (c,t) \mapsto 
                  \begin{cases}
                    \Psi^Y_{k,t}, & \text{if $t>0$}, \\
                    0, & \text{if $t=0$},
                  \end{cases} \\
    \mu^Y_{k,i}   : & \, J_k(Y) \times [0,1] \to  J_k (Y), \quad
                  (c,t) \mapsto 
                  \begin{cases}
                    \Psi^Y_{k,i,t} s_{k-1,i} (\partial_t \Psi^Y_{k-1,t})c, & \text{if $t>0$}, \\
                    0, & \text{if $t=0$},
                  \end{cases}
  \end{aligned}   
  \end{displaymath}
  The maps $\mu^{B_n}_k  :  J_k(B_n) \times [0,1] \to  J_k (B_k)$
  and $\mu^{B_n}_{k,i}   :  J_k(B_n) \times [0,1] \to  J_k (B_n) $ are continuous 
  by \cite[Lem.~4.4]{BraPflHAWFSS}. Note that in the proof of that lemma the assumption
  is needed that $X$ is homologically regular, meaning that $X^{k+1}$ and $\Delta^{k+1}_X$ are regularly situated. 
  Since $J_k(B_n) \to J_k(C_n) $ is a topological identification by the open mapping theorem 
  and Proposition \ref{Prop:FrechetStructuresOnFunctionSpaces}, it follows by commutativity of the diagrams 
  \eqref{Dia1} and \eqref{Dia2}  that $\mu^{C_n}_k$ and $\mu^{C_n}_{k,i}$ are continuous as well. 
  But this means that for all $n\in \N$ and $c \in J_k (X)$ 
  \begin{displaymath}
    \lim_{t \searrow 0} \big( \Psi_{k,t} c  \big)_{|C_n}=   \lim_{t \searrow 0} \mu^{C_n}_k (c_{|C_n},t) = 0 
  \end{displaymath}
  and
  \begin{displaymath} 
  \begin{split}
    \lim_{t \searrow 0} \big( K_{k,t} c \big)_{|C_n} \, & = 
    \lim_{t \searrow 0} \, \sum_{i=1}^{k+1} (-1)^{i+1} \, \int_t^1 \mu^{C_n}_{k,i} (c_{|C_n},s) ds = \\
    & = \sum_{i=1}^{k+1} (-1)^{i+1} \, \int_0^1 \mu^{C_n}_{k,i} (c_{|C_n},s) ds \: ,
  \end{split}
  \end{displaymath}
  where the last integral  exists and is continuous in $c \in J_k (X)$ by continuity 
  of $\mu^{C_n}_{k,i}$. The proof is finished.  
\end{proof}

Before we can apply this result to reduce the computation of the Hochschild homology of the global section algebra to
local ones we need to introduce hypercohomology for complexes of sheaves not bounded below. We follow
\cite{WeiGelEDHCH,WeiCHS} and refer to these papers for further details.

Let $\big( \scrD^\bullet , \delta \big)$ be a cochain complex of sheaves on a paracompact topological space $X$. 
Since the category of sheaves on $X$ has enough injectives, there exists  an injective Cartan--Eilenberg
resolution $\scrD^\bullet \longrightarrow  \scrI^{\bullet,\bullet}$ as described in \cite[Sec.~5.7]{WeiIHA}. 
Up to chain homotopy equivalence, the product total complex $\Tot^{\Pi} (\scrI^{\bullet,\bullet})$ does not
depend on the choice of the injective resolution of $\scrD^\bullet$. Applying the global section functor results
in the total complex $\Gamma \left(X, \Tot^{\Pi} \scrI^{\bullet,\bullet}\right) \cong
\Tot^{\Pi} \Gamma \left(X, \scrI^{\bullet,\bullet}\right) $ which again up to chain homotopy equivalence is 
independant of the particular choice of the resolution $\scrD^\bullet \longrightarrow  \scrI^{\bullet,\bullet}$. 
The \emph{hypercohomology} of  $\big( \scrD^\bullet , \delta \big)$ is then defined by
\[
  \mathbb{H}^k \big( X, \scrD^\bullet \big) =
  H^k \left( \Gamma \left( X, \Tot^{\Pi} \scrI^{\bullet,\bullet}\right) \right) \ . 
\] 
Since a fine sheaf on $X$ is acyclic, one can take fine resolutions instead of injective ones which results in
the observation that for a cochain complex  $\big( \scrD^\bullet , \delta \big)$ of sheaves which are fine already,
the hypercohomology is given by the cohomology of the global sections. More precisely, the following
folklore result holds true. 
  
\begin{theorem}
  Let $\big( \scrD^\bullet , \delta \big) $ be a cochain complex of fine sheaves on a paracompact topological
  space $X$. Let
  $\scrH^\bullet (\scrD^\bullet)$ be its cohomology sheaf that is let
  \[
    \scrH^k (\scrD^\bullet) =
    \ker\left( \delta :\scrD^k\to\scrD^{k+1}\right) \big/ \im\left( \delta: \scrD^{k-1}\to\scrD^k\right) \ .
  \]
  Then the global section functor commutes with cohomology and the hypercohomology of $\big( \scrD^\bullet , \delta \big) $
  is given by
  \begin{equation}
  \label{eq:hypercohomology-fine-sheaf-complex}
  \mathbb{H}^k \big( X, \scrD^\bullet \big) = H_k \big( \Gamma (X, \scrD^\bullet) \big) =
  \Gamma \left( X, \scrH^k (\scrD^\bullet) \right) \quad \text{for all } k\in \N \ .
  \end{equation}   
\end{theorem}

\begin{proof}
  It only remains to show the second equality in Eq.~\eqref{eq:hypercohomology-fine-sheaf-complex}.
  Since $\big( \scrD^\bullet , \delta \big) $  and $\big( \scrH^\bullet (\scrD^\bullet),0\big)$ 
  are quasi-isomorphic sheaf complexes, their hypercohomologies coincide. Both cochain complexes consist
  of fine sheaves, hence their hypercohomology is given by the cohomology of global sections and the claim follows.
\end{proof}

The virtue of Weibel's approach is that it allows to define hypercohomology for  cochain complexes not bounded below.
After reindexing $\scrC^k := \scrC_{-k}$ the components of a chain complex of sheaves $\big( \scrC_\bullet, d \big)$
one therefore obtains a cochain complex of sheaves $\big( \scrC^\bullet, d^* \big)$ for which the hypercohomology groups
$\mathbb{H}^k (X,\scrC^\bullet) $ are defined. We put $\mathbb{H}_k (X,\scrC_\bullet) = \mathbb{H}^{-k} (X,\scrC^\bullet)$
and call these the \emph{hyperhomology} groups of $\big( \scrC_\bullet, d \big)$.
Let us apply this idea to the Hochschild homology theory of a broad class of sheaves of algebras which we call
\emph{sheaves of bornological algebras over a differentiable space} $(X,\calC^\infty_X)$. By definition,
a sheaf $\scrA$ defined on $X$ lies in that category if it satisfies the following three conditions.
\begin{enumerate}[(\textup{BA}1)]
\setcounter{enumi}{0}  
\item\label{ite:algsheaf}
  $\scrA$ is a sheaf of algebras over  $\calC^\infty_X$ that is for each open $U\subset X$ the section space $\scrA (U)$ is
  an algebra over the commutative ring $\calC^\infty_X (U)$ and the corresponding restriction morphism
  $\scrA (U) \to \scrA (V)$  for $V \subset U$ open is an algebra morphism over the restriction homomorphism
  $\calC^\infty_X (U)\to \calC^\infty_X (V)$. 
\item The section space $\scrA (U)$ for $U\subset X$ open is a nuclear LF-space and a bornological algebra with respect to
  the bornology of von Neumann bounded sets. In particular this means that multiplication is separately continuous. 
\item For each open $U\subset X$ the action  $\calC^\infty_X (U) \times \scrA (U)\times \scrA (U)$
  is continuous where $\calC^\infty_X (U)$ carries its natural structure of a Fr\'echet algebra. 
\end{enumerate}
Given a sheaf of bornological algebras $\scrA$ over $(X,\calC^\infty_X)$ we call it  
\begin{enumerate}[(\textup{BA}1)]
\setcounter{enumi}{3}  
\item\label{ite:unitalalgsheaf}
  \emph{unital} if the section spaces $\scrA (U)$  together with their restriction morphisms
  $\scrA (U) \to \scrA (V)$ are unital, and
\item\label{ite:hunital}
  \emph{H-unital} if each of the section spaces $\scrA (U)$ is an
  H-unital algebra that is if the Bar complex corresponding to $\scrA (U)$ is acyclic.
\end{enumerate}  

Note that every unital sheaf of bornological algebras is H-unital. Below we will give in
Section \ref{subsec:convolution-algebras} an example of an H-unital but not unital sheaf of bornological algebras.

Assume that $\scrA$ is a sheaf of bornological algebras over $X$ and let $A =\scrA (X)$ be its algebra of global sections. Then use the completed bornological tensor product $\hatotimes$ as defined in \cite[Sec.~2.2]{MeyACH} 
to define for $U\subset X$ open the the space of Hochschild $k$-chains of the 
bornological algebra $\scrA (U)$:
\[
  \scrC_k (\scrA) (U) := \scrA (U)^{\hatotimes k+1} := \underbrace{\scrA (U) \hatotimes \ldots \hatotimes \scrA (U)}_{k+1\text{\,-times}} \ .
\]
Associating to every open $U\subset X $ the space $\scrC_k (\scrA) (U)$
gives rise to a presheaf on $X$.
The Hochschild boundary operators $b: \scrC_k (\scrA) (U) \to \scrC_{k-1} (\scrA) (U)$ 
combine to a morphism of presheaves $b: \scrC_k (\scrA) \to \scrC_{k-1} (\scrA)$.
After passing to the sheafifications $\hscrC_k (\scrA)$ we therefore obtain a chain complex of sheaves
$\big(\hscrC_\bullet (\scrA),b \big)$ which we call the \emph{Hochschild chain complex} of $\scrA$.
For the case where $\scrA$ coincides with the structure sheaf $\calO_X$ of a differentiable space $X$ one can interpret
$\big(\hscrC_\bullet (\calO_X),b \big)$  as a ``differentiable version'' of Weibel's sheaf complex of Hochschild chains
of a scheme \cite{WeiCHS}. The \emph{Hochschild homology}
$HH_\bullet (\scrA) $ of the sheaf $\scrA$ on $X$ is now defined as the hyperhomology
$\mathbb{H}_\bullet (X, \hscrC_\bullet (\scrA))$. The question arises how $HH_\bullet (\scrA) $ relates to the
Hochschild homology $HH_\bullet (A)$ of the global section algebra $A := \scrA (X)$
and to the global section space of the homology sheaf $\scrH\!\scrH_\bullet (\scrA):= \scrH_\bullet  \big(\hscrC_\bullet (\scrA)\big)$.
The answer is given  by the next result.

\begin{theorem}\label{thm:isomorphism-homology-localization-hyperhomology}
  Let $\scrA$ be a sheaf of H-unital bornological algebras over the homologically regular differentiable space $(X,\calO)$.
  Then the Hochschild chain complex  $\big(\hscrC_\bullet (\scrA),b \big)$ is a complex of fine sheaves
  on $X$. The global section space of its homology sheaf coincides with the Hochschild homology 
  of the global section algebra $A = \Gamma (X,\scrA)$. More precisely, for all $k\in \N$
  \begin{equation}
    \label{eq:Hochschild-homology-commutes-global-sections}
    HH_k (A) = HH_k (\scrA) =HH_k \left( \Gamma (X,\hscrC_\bullet (\scrA))\right) =
    \Gamma \left( X,\scrH\!\scrH_k (\scrA)\right)  \ .
  \end{equation}
\end{theorem}

Before we can prove the theorem, we need to extend the diagonal complex to the situation where 
a sheaf of bornological algebras is given. To define the diagonal complex associated to such a sheaf $\scrA$
choose a singular chart $(\varphi,\varphi^\sharp): (U,\calO|_U) \to (\R^n, \calC^\infty_{\R^n})$ of the
differentiable space $(X,\calO)$ and put for each $k\in \N$
\begin{equation}
  \label{eq:flat-sections-diagonal-sheaf-bornological-algebra}
  \scrJ_k (U):= \scrJ_k^{\varphi} (U) :=
  \varphi^\sharp_{\R^{kn}} \big( \calJ^\infty_{\Delta^{k+1}_{\R^n}} (\R^{kn}) \big) \cdot \hscrC_k (\scrA) (U)
  \subset  \hscrC_k (\scrA) (U)  \ .
\end{equation}
One immediately checks that assigning to an open $V\subset U$ the section space
$\scrJ_k^{\varphi}(V)$ gives rise to a sheaf $\scrJ_k^{\varphi}$ on $U$. Moreover, using transition maps between
singular charts one shows that for a second singular chart
$(\psi,\psi^\sharp): (U,\calO|_U) \to (\R^m, \calC^\infty_{\R^m})$ of $X$ over $U$ the sheaves
$\scrJ_k^\varphi$ and $\scrJ_k^\psi$ over $U$ coincide. Hence they glue together to a global sheaf over
$X$ which we denote by $\scrJ_k$ and which by construction is a subsheaf of $\hscrC_k (\scrA)$.
As in the proof of Lemma \ref{Lem:HochschildBdryMapsIdealToIdeal} one shows that the Hochschild boundary $b$ maps
$\scrJ_k$ to $\scrJ_{k-1}$ hence we obtain a monomorphism of sheaf complexes
$\big( \scrJ_\bullet  , b\big) \hookrightarrow \big( \hscrC_\bullet (\scrA) , b  \big)$.  

\begin{proposition}\label{prop:isomorphism-hochschild-chain-complex-diagonal-complex}
  Assume to be given a sheaf of H-unital bornological algebras $\scrA$ over the
  homologically regular differentiable space $(X,\calO)$.
  Let $\big( \scrE_\bullet  , \beta \big)$ be its \emph{diagonal sheaf complex}
  that means let it be the unique quotient making the following  sequence of chain complexes of sheaves exact: 
  \begin{equation}
    0 \longrightarrow \big( \scrJ_\bullet  , b\big) \longrightarrow \big( \hscrC_\bullet (\scrA) , b  \big)
    \overset{p}{\longrightarrow} \big( \scrE_\bullet  , \beta \big)  \longrightarrow 0 \: .
  \end{equation}
  Denote by $A$  the global section algebra $\Gamma (X,\scrA)$.
  Then the canonical chain map $C_\bullet (A) \to \hscrC_\bullet (\scrA) (X)$  
  and the projection $p_X :\hscrC_\bullet (\scrA) (X) \to E_\bullet (A) = \scrE_\bullet (X)$ are quasi-isomorphisms.
\end{proposition}
\begin{proof}
  Proposition 2.7 in \cite{PflPosTanHochschild} says that the chain map $C_\bullet (A) \to \hscrC_\bullet (\scrA) (X)$ is
  a quasi-isomorphism, so it remains to show that the
  projection $p_X :\hscrC_\bullet (\scrA) (X) \to E_\bullet (A)$ is a quasi-isomorphism. 
  We only sketch the argument because it is analogous to the proof of
  Theorem \ref{thm:projection-diagonal-complex-qism}.

  \textit{Case 1.} Let us first consider the local case that is the case where $X$ is assumed to be a closed
  subspace of some Euclidean space $\R^n$. The global section algebra $\calO (X)$ then coincides
  with the quotient of $\calC^\infty (\R^n)$ by a closed ideal $I$, and the space $X$ can be identified
  with  the real spectrum of the Fr\'echet algebra $\calO (X) = \calC^\infty (\R^n)/I$.
  Recall the definition of the cutoff functions $\Psi_{k,i,t} \in \calC^\infty(\R^{n(k+1)})$
  for $k\in \N$, $i=0,\ldots,k$ and $t>0$ by  Eq.~\ref{eq:cutoff-functions-diagonal}. Again put 
  $\Psi_{k,t} := \Psi_{k,k,t}$. 
  We will regard the functions $\Psi_{k,i,t}$ and $\Psi_{k,t}$ both as smooth functions on $\R^{n(k+1)}$
  and as elements of $\calO (X)^{\hatotimes (k+1)} =  \calO(X^{k+1})$. 

  \textit{Case 1a.}  After these preparations, we assume that $\scrA$ is a sheaf of unital bornological
  algebras over $(X,\calO)$. Then the degeneracy maps $s_{k,i}: C_k (A) \to C_{k+1}(A)$ defined
  by Eq.~\ref{eq:degenarcy-maps} exist for the global section algebra $A$ and extend to
  sheaf morphisms $s_{k,i}: \hscrC_k (\scrA) \to \hscrC_{k+1}(\scrA)$.
  By the natural action of $\calO(X^{k+1})$ on $C_k (A)$ we obtain for each $t>0$ a chain map
  \[
   \Psi_{\bullet,t} : C_\bullet (A) \to C_\bullet(A) , \:  C_k (A) \ni c \mapsto \Psi_{k,t} \cdot c
  \]
  as in \emph{Step 2} of the proof of Theorem \ref{thm:projection-diagonal-complex-qism}.
  \emph{Steps 3} to \emph{5} in that proof also extend. 
  In particular we have $\calO(X)$-module maps
  $\eta_{k,i,t}: C_k (A)\to C_{k+1} (A)$ defined by Eq.~\ref{eq:homotopy-components} and a graded map 
  $H_{\bullet,t} : C_\bullet (A) \to C_\bullet(A)$ of degree $1$ defined by
  $H_{k,t} = \sum_{i=1}^{k+1} (-1)^{i+1} \eta_{k,i,t}$. Since Eq.~\ref{EqAlgHom} then holds 
  as well, $H_{\bullet,t}$ is a homotopy between the identity and the localization morphism $\Psi_{\bullet,t}$.
  Next define $K_{k,t}: C_k(A) \to C_{k+1}(A)$ as in Eq.~\ref{eq:homotopies-between-t-values}
  and observe that like in \emph{Step 4} one proves that $K_{\bullet,t}$
  provides a homotopy between the localizations $\Psi_{\bullet,1}$ and $\Psi_{\bullet,t}$.
  Finally we want to show that the limit $\lim_{t\searrow 0} K_{\bullet,t}$ exists and provides a homotopy between
  the localization $\Psi_{\bullet,1}$ and $0$.
  To this end recall from \emph{Step 5} in the proof of
  Theorem \ref{thm:projection-diagonal-complex-qism}  that for
  $f \in \scrJ^\infty (\Delta^{k+1}_{\R^n}, \R^{(k+1)n})$ the limit  $\lim_{t\searrow 0} \Psi_{k,t} f$
  exists and vanishes. Now assume to be given $c \in \scrJ_k(X)$ of the form
  $c =\sum_{i=1}^l f_i c_i$ with $f_i \in \scrJ^\infty (\Delta^{k+1}_{\R^n}, \R^{(k+1)n})$ and
  $c_i \in \hscrC_k (\scrA) (X)$. Then  $\lim_{t\searrow 0} \Psi_{k,t} c = 0$
  and the limit $K_k c = \lim_{t\searrow 0} K_{k,t} c $
  exists. The latter is even continuous in $c$. Since the space of $c\in \scrJ_k(X)$ of the given form
  is dense in $ \scrJ_k(X)$ by Eq.~\ref{eq:flat-sections-diagonal-sheaf-bornological-algebra},
  the operator  
  \[
    K_k : \scrJ_k(X) \to \scrJ_{k+1} (X), \: c \mapsto \lim_{t\searrow 0} K_{k,t}c 
  \]
  therefore is well-defined and continuous. By construction, it is a homotopy between the localization $\Psi_{\bullet,1}$
  and $0$. Hence the chain complex $\scrJ_\bullet (X)$ is acyclic, and $p_X$ is a quasi-isomorphism as claimed.

  \textit{Case 1b.} Now we come to the case where $\scrA$ is not necessarily unital, but still assumed to be
  an H-unital bornological sheaf of algebras. Consider the direct sum $\scrA \oplus \calO$ and denote it by
  $\scrA^+$. Since $\scrA$ is a $\calO$-module sheaf and $\calO$ is a sheaf of unital algebras,
  $\scrA^+$ inherits in antural way the structure of a sheaf of unital algebras such that the following sequence
  of sheaves is split exact.
  \begin{equation*}
    0 \longrightarrow \scrA \longrightarrow \scrA^+ \longrightarrow \calO \longrightarrow 0
  \end{equation*}
  By \emph{Case 1a}, the chain complex $\hscrC_\bullet (\scrA^+) (X)$ is
  quasi-isomorphic to the diagonal complex $E_\bullet (A^+) = \hscrC_\bullet (\scrA^+) (X) / \scrJ_\bullet^+(X)$,
  where $A^+ =\scrA^+ (X)$ and where the sheaf complex $\scrJ_\bullet^+$ is defined by
  Eq.~\ref{eq:flat-sections-diagonal-sheaf-bornological-algebra} with $\scrA$ replaced by $\scrA^+$.
  Since $A =\scrA (X)$ is H-unital and $E_\bullet (A^+) \cong E_\bullet (A) :=  \hscrC_\bullet (\scrA) (X)/ \scrJ_\bullet(X)$,
  the chain complex $\hscrC_\bullet (\scrA) (X)$ has to be 
  quasi-isomorphic to the diagonal complex $E_\bullet (A)$ which finalizes \emph{Case 1}. 
  See \cite[Lemma 2.4.]{PflPosTanHochschild} for further details in the last argument.

  \textit{Case 2.} It remains to prove the claim when $(X,\calO)$ is not necessarily affine. We only consider
  the unital case, the H-unital one can be treated is as in \emph{Case 1b}.
  Choose a family $(\varphi_i,\varphi_i^\sharp)_{i\in I}$ of singular charts of $(X,\calO)$
  such that the family of its domains $(U_i)_{\in I}$ is a locally finite cover of $X$ and such that each $U_i$ is
  relatively compact. Let $\R^{n_i}$ be the range of $\varphi_i $.
  Choose a partition of unity $(\psi_i)_{i\in I}$ subordinate partition subordinate to the cover $(U_i)_{i\in I}$ 
  and such that the support of each cutoff function $\psi_i$ is compact.
  Recall from \emph{Step 2} in the proof of Theorem \ref{thm:projection-diagonal-complex-qism} that the cutoff functions
  $\Psi_{k,t}$ can be globally defined on $X$ by embedding $X$ into the colimit  $(\R^\infty,\calC^\infty_{\R^\infty})$.
  Now fix a degree $k$ and choose $\varepsilon_i > 0$ such that 
  $\supp \Psi_{k,\varepsilon_i} \cap \supp \varphi \times X^k $ is relatively compact in $U_i^{k+1}$.  
  By \emph{Case 1}, there exists for each $i$ a contracting homotopy
  $H_{\bullet}^i : \scrJ^{\varphi_i}_\bullet (U_i) \to \scrJ^{\varphi_i}_\bullet (U_i) $.
  Put for $l\leq k$
  \[
    H_l = \sum_{i\in I} H_{l}^i \varphi_i \Psi_{l,\varepsilon_i} : \scrJ_k (X) \to \scrJ_{k+1} (X) \ . 
  \]
  The equality 
  \[
    b H_l c+ H_l b c= c   
  \] 
  then holds for all $c\in  \scrJ_l (X)$ and $l\leq k$, hence $\scrJ_\bullet (X)$ is contractible and $p_X$
  a quasi-isomorphism.    
\end{proof}


\section{Applications}
In this section we provide several examples of differentiable spaces and sheaves of bornological
algebras where our method of localization can be applied and used to determine the corresponding
Hochschild homology groups.

\subsection{Whitney fields}\label{sec:whitney-fields}
Let $X \subset \R^n$ be a regular locally closed subset with regularly
situated diagonals that is $X^k$ and $\Delta^k_{\R^n}$ are regularly situated for all positive $k$.
Denote for every relatively open $U\subset X$
by $\calE^\infty (U)$ the algebra of Whitney fields over $U$. Then $(X,\calE^\infty)$ is a
differentiable space with global sections being the Whitney fields over $X$.
For each $x\in X$, the stalk of the sheaf $\hscrC_\bullet (\calE)$ at $x$ coincides with
the tensor product of the stalk $\calE^\infty_x$ with the Hochschild chain complex
$C_\bullet (\calF_n)$ of the algebra $\calF_n$ of formal power series in $n$ variables. 
The Hochschild homology of $\calF_n$ coincides in degree $k$ with $\calF_n \otimes \Lambda^k\R^n$
as one shows for example by using the Connes--Koszul resolution for $\calF_n$. By the
localization result Theorem \ref{thm:isomorphism-homology-localization-hyperhomology}
this entails that the Hochschild homology of $\calE^\infty (X)$ is given by
\[
  HH_k\big( \calE^\infty (X)\big) \cong \Omega^k_{\calE^\infty} (X) := \calE^\infty(X) \otimes \Lambda^k\R^n \ . 
\]
We thus have recovered \cite[Thm.~4.8]{BraPflHAWFSS}. 

\subsection{Convolution algebras}
\label{subsec:convolution-algebras}
Let $\sfG \rightrightarrows M$ be a proper Lie groupoid. Denote by $s,t$ its source and target map, respectively,
and by $\pi : M \to X:= M/\sfG $ the canonical projection onto the space of orbits. Note that $X$ is a differentiable 
space by \cite{PflPosTanGOSPLG}. Its structure sheaf $\calC^\infty_X$ has section spaces $\calC^\infty_X (U)$ over $U\subset X$ open
given by all continuous functions $f:U\to \R$ such that the pullback $\pi^*f$ is smooth on $M$. 
Given a subset $C \subset X$ put $C^0 := \pi^{-1} (C)$ and $C^1 := s^{-1} (C^0)$ 
and call them the saturations of $C$ in $M$ and $\sfG$, respectively. 
A subset $K\subset \sfG$ then is said to be \emph{longitudinally compact} if  
for every compact $C\subset X$ the intersection $K$ with the saturation $C_1$ is compact in $\sfG$.
For each open $U\subset X$  we now define the \emph{convolution algebra} of $\sfG$ over $U$ by 
\[
 \scrA (U) = \{ f\in \calC^\infty (U_1) \mid 
 \supp (f) \text{ is longitudinally compact} \} \ .
\]
The product on $\scrA (U)$ is the \emph{convolution product} given by
\[
  f_1 \ast f_2 (g) = \int_{\sfG(t(g),-)} f_1(h) f_2(h^{-1}g) d\lambda^{t(g)} ,
  f_1,f_2\in \scrA (U), \: g\in U^1,
\]
where $(\lambda^x)_{x\in M}$ is a smooth left Haar system on $\sfG$.
One verifies that one thus obtains a sheaf $\scrA$ of H-unital bornological algebras over the orbit space $X=M/\sfG$
\cite{PflPosTanHochschild}. The sheaf $\scrA$ is called the \emph{convolution sheaf} of the groupoid $\sfG$ and is sometimes denoted more descriptively by $\scrA_\sfG$.
Its algebra of global sections  $A = \Gamma (X,\scrA_\sfG )$ is the
\emph{smooth convolution algebra} of $\sfG$. The localization result
Theorem \ref{thm:isomorphism-homology-localization-hyperhomology} applies
and gives rise to an isomorphism 
\[
   HH_\bullet (A) = \Gamma \left( X,\scrH\!\scrH_\bullet (\scrA_\sfG)\right) \ .
\]
The stalks of the sheaf $\hscrC_\bullet (\scrA)$ of Hochschild chains have been identified
in \cite[Sec.~3]{PflPosTanHochschild}. More precisely, it has been shown in
\cite[Prop.~3.5]{PflPosTanHochschild} that for every orbit
$x = \sfG p \in X$ through a point $p\in M$ there exists a quasi-isomorphism
\begin{equation}
  \label{eq:qusims-stalks-convolution-algebras-action-groupoids}
  L_{\bullet,x} : \hscrC \big(\calA_\sfG\big)_{\bullet,x} \to
  \hscrC \big(\calA_{G_p\ltimes N_p}\big)_{\bullet,0} \ ,
\end{equation}
where $G_p$ denotes the isotropy group of $\sfG$ at $p$, $N_p$ its normal space, and
$\calA_{G_p\ltimes N_p}$ the convolution sheaf of the action groupoid $G_p\ltimes N_p$.
Brylinski conjectured in \cite{BryAAGAH,BryCHET} that the Hochschild homology of the
convolution algebra $\calA_{G\ltimes M}$ of a compact Lie group action on $M$ coincides with the so-called
\emph{basic relative forms} on the inertia space $\Lambda_0 (G\ltimes M)$.
The conjecture has been proved for $S^1$-actions in \cite[Thm.~6.10]{PflPosTanHochschild}, but even
though the conjecture is believed to be correct a complete proof is still lacking.
Together with the quasi-isomorphisms
\ref{eq:qusims-stalks-convolution-algebras-action-groupoids}, verification of Brylinski's
conjecture would allow to determine the Hochschild and possibly even the cyclic homology of
proper groupoids.

\subsection{Differentiable stratified spaces}

Let $(X,\calC^\infty)$ be a homologically regular differentiable stratified space
and assume that $X$ is endowed with a (minimal) Whitney stratification compatible with the
differentiable structure $\calC^\infty_X$ which means that for every stratum $S\subset X$ the canonical
embedding $S\hookrightarrow X$ or in other words that
$(S,\calC^\infty_S)\hookrightarrow (X,\calC^\infty_X)$ is an embedding of differentiable spaces.
Examples of such spaces are real or complex algebraic varieties, real analytic varieties,
semialgebraic or subanalytic sets, orbit spaces of compact Lie group actions and orbit spaces of proper Lie groupoids. Unless $X$ is smooth, the Hochschild homology theory of the global section algebra
$A=\calC^\infty (X)$ is in general not known. Localization methods as those in
Section \ref{sec:diagonal-complex} appear to be helpful for making progress in understanding the
cyclic homology theory of such space and for obtaining at least a qualitative picture. 
By Theorem \ref{thm:projection-diagonal-complex-qism}, the Hochschild homology of $A$ can be computed
by the homology of the diagonal complex $E_\bullet (A)$.
The homology of $E_\bullet (A)$ coincides with the global section space of the homology sheaf
$\scrH_\bullet \big(\scrE_\bullet (\calC^\infty_X)\big)$. Denote by $X^\circ\subset X$ the 
subspace consisting of the union of all strata of depth $0$. Then $X^\circ$ is open, dense and
smooth, hence its Hochschild homology is known to be isomorphic to the graded space of
differential forms on $X^\circ$ that is
\[
  HH_\bullet (\calC^\infty (X^\circ)) \cong
  \scrH_\bullet \big(\scrE_\bullet (\calC^\infty_{X^\circ})\big) \cong
  \Omega^\bullet (X^\circ) \ . 
\]
The question arises, how the Hochschild homology sheaf
$\scrH\scrH_\bullet (\calC^\infty_X) := \scrH_\bullet \big(\hscrC (\calC^\infty_X)\big)$
looks over a stratum $S$ in the singular locus $X_{\textup{sing}}$ of $X$.
The following appears to be a reasonable conjecture. 
\begin{conjecture}
  Assume to be given a homologically regular reduced differentiable stratified space $(X,\calC^\infty)$, possibly
  from an allowable class like locally semialgebraic differentiable spaces or
  orbit spaces of proper Lie groupoids. 
  Then the Hochschild homology sheaf $\scrH\scrH_\bullet (\calC^\infty_X)|_S$ restricted
  to a stratum $ S\subset X$ is locally free, and the strata are maximal
  subsets of $X$ with that property. 
\end{conjecture}
If the conjecture holds true, the stratification theory of a differentiable space from
an allowable class could be detected by its Hochschild homology. In particular, Hochschild homology
would serve as an invariant of singularity types. To the authors knowledge this observation
has not been made in cyclic homology theory yet and could lead to significant progress in the
understanding of singularities from the point of view of noncommutative geometry. 
Strong indication for that philosophy is the work by Michler \cite{MicTDAQHH}, where it has been
shown that for complex isolated singularities in quasi-homogeneous complex hypersurfaces the
Milnor number is recovered by its cyclic homology. 
We have another related conjecture. 

\begin{conjecture}
  The diagonal sheaf complex $\scrE_\bullet (\calC^\infty_X)$ and the Hochschild homology sheaf
  $\scrH\scrH_\bullet (\calC^\infty_X)$
  associated to a homologically regular reduced
  differentiable space $(X,\calC^\infty)$ are quasi-coherent. 
\end{conjecture}

For complex analytic spaces, Palamodov has shown in \cite{PalIDQCAS}  that the 
the corresponding analytic Hochschild homology sheaf is coherent which seems to substantiate this
conjecture. 

\subsection{Localization in cyclic homology}
Given a differentiable space $(X,\calO)$ with global section algebra $A= \calO (X)$,
Connes' boundary operator $B : C_k (A) \to C_{k+1} (A)$ does not commute with the
localization maps $\Psi_{\bullet,t}$, $t>0$. Hence, localization in Hochschild homology
cannot be directly transferred to Connes' $(b,B)$-complex which determines cyclic homology.
But the situation is not hopeless since Connes' spectral sequence might still degenerate.
For example, in the case where $A$ is the algebra $\calE^\infty (X)$ of Whitney fields as in
Section \ref{sec:whitney-fields}, the spectral sequence degenerates at the $E^1$-term which allows
for the computation of the cyclic homology of $\calE^\infty (X)$ as in Connes'
original paper \cite{ConNDG}. This approach has been pursued in \cite{BraPflHAWFSS}, and the cyclic homology of
Whitney fields could be identified with
\[
  HC_k\big( \calE^\infty (X)\big) \cong \Omega^k_{\calE^\infty} (X) / d \Omega^{k-1}_{\calE^\infty} (X)
  \oplus H^{k-2}_\textup{WdR} (X) \oplus H^{k-4}_\textup{WdR} (X) \oplus  \ldots \ ,  
\]
where $H^k_\textup{WdR} (X)$ is the $k$-th \emph{Whitney-de Rham cohomology group} of $X$ that is the
$k$-th cohomology group of the cochain complex
$\big( \Omega^\bullet_{\calE^\infty} (X) , d \big) := \big(\calE^\infty (X) \otimes \Lambda^\bullet \R^n , d \big)$.
Note, however, that Connes' spectral sequence need not always degenerate. This has been pointed
out e.g.~in \cite[Rem.~3.8]{BryAAGAH} referring to an example by Block. 

Another approach to localization in cyclic homology comes from the observation that the
localization maps $\Psi_{\bullet,t}$ commute with the cyclic operator $\lambda$. The homotopy
operators $H_{\bullet,t}$ do not have this property, though. Hence the homotopies $H_{\bullet,t}$
do not descend to Connes' cyclic chain complex, and it is not clear whether $\Psi_{\bullet,t}$ defines a
quasi-isomorphism on the cyclic chain complex. To circumvent this obstacle it appears to be promising  
to consider localization within the cyclic bicomplex directly. Progress in this approach is expected to
eventually lead to unravel the cyclic homology theory of differentiable spaces and Lie groupoids.

\bibliographystyle{amsplain}
\bibliography{lmlib,ncgeometry&groupoids} 
\end{document}